\documentclass[11pt]{amsart}

\usepackage{a4wide}

\usepackage{amssymb}

\usepackage{enumerate}

\usepackage{cite}

\usepackage{hyperref}

\hypersetup{%
pdftitle={},
pdfsubject={Mathematics},
pdfauthor={Manfred Madritsch},
pdfkeywords={}
hyperindex=true,plainpages=false}

\newtheorem{thm}{Theorem}[section]

\newtheorem{lem}[thm]{Lemma}

\theoremstyle{definition}
\newtheorem{defin}[thm]{Definition}
\newtheorem{rem}[thm]{Remark}
\newtheorem{exa}[thm]{Example}

\numberwithin{equation}{section}

\DeclareMathOperator{\diam}{diam}

\allowdisplaybreaks[3]

\newcommand{\N}{\mathbb{N}}

\newcommand{\Q}{\mathbb{Q}}
\newcommand{\R}{\mathbb{R}}

\renewcommand{\lvert}{\left\vert}
\renewcommand{\rvert}{\right\vert}
\renewcommand{\lVert}{\left\Vert}
\renewcommand{\rVert}{\right\Vert}

\title[Non-normal numbers]
{Non-normal numbers with respect to infinite Markov partitions}

\author[M. G. Madritsch]{Manfred G. Madritsch}
\address[M. G. Madritsch]{1. Universit\'e de Lorraine, Institut Elie Cartan de Lorraine, UMR 7502, Vandoeuvre-l\`es-Nancy, F-54506, France;
2. CNRS, Institut Elie Cartan de Lorraine, UMR 7502, Vandoeuvre-l\`es-Nancy, F-54506, France}
\email{manfred.madritsch@univ-lorraine.fr}



\date{\today}

\begin{document}

\begin{abstract}
  In the present paper we investigate two special types of non-normal
  numbers. On the one hand we call a number extremely non-normal if the
  set of accumulation points of its frequency of blocks vector is the
  full set of shift invariant probability vectors. On the other hand
  we call a number having maximal oscillation frequency if for any
  fixed block the set of accumulation points of its frequency vector
  is the full set of possible probability vectors. The goal is to
  investigate the Baire category of these numbers for infinite Markov
  partitions. 
\end{abstract}

\maketitle

\section{Introduction}\label{sec:introduction}
Let $\mathbb{I}$ denote the irrational numbers in the unit interval,
\textit{i.e.}, $\mathbb{I}:=[0,1]\setminus\Q$. Then every
$x\in\mathbb{I}$ can be represented in a unique way as an infinite
simple continued fraction, namely
\[
x=[a_1(x),a_2(x),a_3(x),\ldots]=\cfrac{1}{a_1(x)+\cfrac{1}{a_2(x)+\cfrac{1}{a_3(x)+\ddots}}},
\]
where $a_k(x)\in\N$ for $k\geq1$.

In the present paper we investigate the limiting frequency of blocks
of $a_k(x)$. For $k\geq1$ a positive integer and a block $\mathbf{b}=b_1\ldots
b_k\in\N^k$ we denote by $\Pi(x,\mathbf{b},n)$ the number of
occurrences of this block $\mathbf{b}$ among the first $n$ digits of
$x$, \textit{i.e.}
\[
\Pi(x,\mathbf{b},n):=\frac{\{0\leq i<n:a_{i+1}(x)=b_1,\ldots,a_{i+k}(x)=b_k\}}{n}.
\]
Furthermore we denote by
\[
  \Pi_k(x,n)=\left(\Pi(x,\mathbf{b},n)\right)_{\mathbf{b}\in\N^k}
\]
the vector of frequencies $\Pi(x,\mathbf{b},n)$ for all blocks
$\mathbf{b}$ of length $k$.

For digits (blocks of length 1) a famous result of L\'evy \cite{levy1929:sur_les_lois} states that for Lebesgue almost
all $x\in\mathbb{I}$ we have
\begin{gather}\label{mani:levy}
\Pi(x,b,n)\rightarrow\frac1{\log 2}\log\frac{(b+2)^2}{b(b+2)}
\end{gather}
for all $b\in\N$. In analogy with normal numbers in $q$-adic number
systems (\textit{cf.}
\cite{Kuipers_Niederreiter1974:uniform_distribution_sequences} or
\cite{drmota_tichy1997:sequences_discrepancies_and}) we call a
$x\in\mathbb{I}$ simple (continued-fraction-)normal if it satisfies
\eqref{mani:levy}.

We can extend this notion to (continued-fraction-)normal numbers by
using the Gauss measure defined
by $$\mu(A)=\frac1{\log2}\int_A\frac1{1+x}\mathrm{d}x,$$ where
$A\subset[0,1]$ is a Lebesgue measurable set. Then we call a number
(continued-fraction-)normal if the asymptotic frequency of its block
of digits is determined by the Gauss measure. Now an application of
Birkhoff's ergodic theory (\textit{cf.}
\cite{billingsley1965:ergodic_theory_and} or
\cite{dajani_kraaikamp2002:ergodic_theory_numbers}) yields that
allmost all numbers are (continued-fraction-)normal with respect to
the Lebesgue measure.

In the present paper we want to focus on numbers which are
described by their frequency of blocks of digits. For each
$k\geq1$, we define the simplex of all probability vectors $\Delta_k$ by
\[
  \Delta_k=\left\{(p_{\mathbf{i}})_{\mathbf{i}\in\N^k}:p_{\mathbf{i}}\geq0,
  \sum_{\mathbf{i}\in\N^k}p_{\mathbf{i}}=1\right\}.
\]
The set $\Delta_k$ together with the $1$-norm $\lVert\cdot\rVert_1$ defined by
$$\lVert \mathbf{p}-\mathbf{q}\rVert_1=
\sum_{\mathbf{i}\in\N^k}\lvert p_\mathbf{i}-q_\mathbf{i}\rvert.$$
is a compact metric space.

On the one hand we clearly have that any vector $\mathbf{\Pi}_k(\omega,n)$ of frequencies of
blocks of digits of length $k$ belongs to $\Delta_k$. On the other
hand any probability vector needs to be shift invariant (\textit{cf.}  Volkmann
\cite{volkmann1958:uber_hausdorffsche_dimensionen6} or Olsen
\cite{olsen2004:extremely_non_normal,olsen2003:extremely_non_normal}). Therefore
we define the
subsimplex of shift invariant probability vectors $\mathrm{S}_k$ by
\[
\mathrm{S}_k:=\left\{(p_{\mathbf{i}})_{\mathbf{i}\in\N^k}:
p_{\mathbf{i}}\geq0,\sum_{\mathbf{i}\in\N^k}p_{\mathbf{i}}=1,
\sum_{i\in\N}p_{i\mathbf{i}}=\sum_{i\in\N}p_{\mathbf{i}i}\text{ for
  all }\mathbf{i}\in\N^{k-1}\right\}.
\]

Now we want to focus on two different examples of non-normal sets. Our
first one is the set of extremely non-normal numbers $\mathbb{E}$
introduced by Olsen \cite{olsen2003:extremely_non_normal}. We call a
number $x\in\mathbb{I}$ extremely non-$k$-normal if each probability
vector $\mathbf{p}\in\mathrm{S}_k$ is an accumulation point of the
sequence of vectors of frequencies
$(\Pi_k(x,n))_{n\in\N}$. Furthermore we call $x\in\mathbb{I}$
extremely non-normal if it is extremely non-$k$-normal for every
$k\geq1$. Then Olsen \cite{olsen2003:extremely_non_normal} could prove
the following
\begin{thm}[{\cite[Theorem 1]{olsen2003:extremely_non_normal}}]
The set of extremely non-normal numbers is residual.
\end{thm}

Our second example are continued fractions with maximal frequency
oscillation. We denote by $\mathrm{A}(x,\mathbf{b})$ the set of all
accumulation points of the sequence
$(\Pi(x,\mathbf{b},n))_{n\in\N}$. Furthermore we set
$$\mathrm{A}(\mathbf{b})=\bigcup_{x\in\mathbb{I}}\mathrm{A}(x,\mathbf{b}).$$
Then the set of continued fractions with maximal frequency oscillation
$\mathbb{F}$ is defined by
$$\mathbb{F}=\{x\in\mathbb{I}\colon
\mathrm{A}(\mathbf{b})=\mathrm{A}(x,\mathbf{b})\text{ for all }\mathbf{b}\}.$$

Liao, Ma and Wang \cite{liao_ma_wang2008:dimension_some_non} could
show the following
\begin{thm}[{\cite[Theorem 1.1]{liao_ma_wang2008:dimension_some_non}}]
\label{F_residual}
The set of continued fractions with maximal frequency oscillation is residual.
\end{thm}

These examples explain the dichotomy between normal and non-normal
numbers. Since Lebesgue almost all numbers are normal, they are
natural elements in the sense of probability. Whereas the set of
non-normal numbers is residual and therefore they are natural elements
in the sense of topology. This dichotomy has been exploit amongst
others for the $q$-adic number systems by Albeverio, Hyde, Laschos,
Olsen, Petrykiewicz, Pratsiovytyi, Shaw, Torbin and Volkmann in
\cite{albeverio_pratsiovytyi_torbin2005:topological_and_fractal,
  albeverio_pratsiovytyi_torbin2005:singular_probability_distributions,
  hyde10:_iterat_cesar_baire, olsen2004:extremely_non_normal,
  volkmann1958:uber_hausdorffsche_dimensionen6}, for the continued
fractions by Olsen \cite{olsen2003:extremely_non_normal} and for
iterated function systems by Beak and Olsen
\cite{baek_olsen2010:baire_category_and}.

In the present paper we want to investigate two main extensions of
these examples. First we want to generalize the continued fraction
expansions to infinite Markov partions. A similar step has already
been made for finite Markov partitions in another paper by the author
\cite{madritsch2014:non_normal_numbers}. Motivated by considerations
of Hyde \textit{et al.} \cite{hyde10:_iterat_cesar_baire} we want to
investigate the Baire category of Ces\`aro variants of the sets of
extremely non-normal numbers and numbers with maximal oscillation
frequency.


\section{Definitions and statement of results}
We want to start by taking a step back and consider general dynamical
systems. In our notation we mainly follow Chapter 6 of
\cite{lind_marcus1995:introduction_to_symbolic}. Let $M$ be a metric
space and $\phi:M\to M$ be a continuous map, then we call the pair
$(M,\phi)$ a dynamical system. As above we need some ``digits''. Let
$\mathcal{P}:=\{P_1,P_2,\ldots\}$ be an infinite family of disjoint
open sets. Then we call $\mathcal{P}$ a topological partition (of $M$)
if $M$ is the union of the closures $\overline{P_i}$ for $i\geq1$,
\textit{i.e.}
\[
M=\bigcup_{i\in\N}\overline{P_i}.
\]

For the rest of the paper let us assume that a dynamical system
$(M,\phi)$ together with an infinite topological partition $\mathcal{P}$ is
given. Now we want to take a closer look at the underlying symbolic
dynamical system. Without loss of generality we may denote by $\N$ the
alphabet corresponding to the partition $\mathcal{P}$. Furthermore let
$\N^k$ be the set of words of length $k$ and
\[
\N^*=\bigcup_{k\in\N}\N^k
\]
be the set of finite words. Finally we denote by $\N^\N$ the set of
infinite words over $\N$.

For an infinite word $\omega=\omega_1\omega_2\omega_3\ldots\in\N^\N$
and a positive integer $n$ we denote by $\omega\vert
n=\omega_1\omega_2\ldots\omega_n$ its truncation to the $n$th
place. Furthermore for a given finite word $\omega\in\N^*$ we denote
by $[\omega]\subset\N^\N$ the cylinder set consisting of all infinite
words starting with the same letters as $\omega$, \textit{i.e.}
\[
[\omega]:=\left\{\gamma\in\N^\N:\gamma\vert\lvert\omega\rvert=\omega\right\}.
\]

Now we want to describe the shift space that is generated. Therefore
we call an infinite word $\omega=a_1a_2a_3\ldots\in\N^\N$ allowed for
$(\mathcal{P},\phi)$ if
\[
  \bigcap_{k=1}^\infty\phi^{-k}\left(P_{a_k}\right)\neq\emptyset.
\]
Let $\mathcal{L}_{\mathcal{P},\phi}$ be the set of allowed words. Then
$\mathcal{L}_{\mathcal{P},\phi}$ is a language and there is a unique shift
space $X_{\mathcal{P},\phi}\subseteq\N^\N$, whose language is
$\mathcal{L}_{\mathcal{P},\phi}$.
We call $X_{\mathcal{P},\phi}\subseteq\N^\N $ the one-sided
symbolic dynamical system corresponding to
$(\mathcal{P},\phi)$. Finally for each $\omega=a_1a_2a_3\ldots
\in X_{\mathcal{P},\phi}$ and $n\geq0$ we denote by $D_n(\omega)$ the
cylinder set of order $n$ corresponding to $\omega$ in $M$,
\textit{i.e.},
\[
D_n(\omega):=\bigcap_{k=0}^n\phi^{-k}(P_{a_k})\subseteq M.
\]
Now we can state the definition of an infinite Markov partition.
\begin{defin}
  Let $(M,\phi)$ be a dynamical system and
  $\mathcal{P}=\{P_1,P_2,P_3,\ldots\}$ be an infinite topological partition of
  $M$. Then we call $\mathcal{P}$ an infinite Markov partition if the
  generated shift space $X_{\mathcal{P},\phi}$ is of finite type and
  for every $\omega\in X_{\mathcal{P},\phi}$ the intersection
  $\bigcap_{n=0}^\infty\overline{D_n(\omega)}$ consists of exactly one
  point.
\end{defin}

After introducing all the necessary ingredients we want to link the
introduced concept of infinite Markov partitions with the continued
fraction expansion and L\"uroth series (\textit{cf.} Dajani and
Kraaikamp \cite{dajani_kraaikamp2002:ergodic_theory_numbers}).

\begin{exa}\label{example:continued_fractions}
The dynamical system $([0,1],T)$, where $T$ is the Gauss map
\begin{gather*}
Tx=\begin{cases}\frac1x-\lfloor\frac1x\rfloor&\text{for }x\neq0,\\
0&\text{for }x=0,\end{cases}
\end{gather*}
together with the infinite topological partition
\[\mathcal{P}:=
\left\{\left(\tfrac12,1\right), \left(\tfrac13,\tfrac12\right),
  \left(\tfrac14,\tfrac13\right), \ldots\right\}\]
provides the continued fraction expansion.

In particular, if we set $x_1:=x$ and $x_{k+1}:=Tx_k$ for $k\geq1$ then
$a_k(x)=\lfloor\frac1{x_k}\rfloor$ for $k\geq1$ is the continued
fraction expansion of $x$ from the introduction.
\end{exa}

There are several extensions of the continued fraction expansion like
continued fractions from below, nearest integers continued fractions,
$\alpha$-continued fractions, Rosen continued fractions and
combination of those. For criteria such that a number is normal with
respect to different continued fractions expansions we refer the
reader to a paper by Kraaikamp and Nakada
\cite{kraaikamp_nakada2000:normal_numbers_continued}.

\begin{exa}\label{example:lueroth}
Let $\phi:[0,1)\to[0,1)$ be defined by
$$\phi(x)=\begin{cases}
n(n+1)x-n, &x\in\left[\frac1{n+1},\frac1n\right),\\
0 &x=0.
\end{cases}.$$ Then the pair $([0,1),\phi)$ together with the infinite
topological partition
\[\mathcal{P}:=
\left\{\left(\tfrac12,1\right), \left(\tfrac13,\tfrac12\right),
  \left(\tfrac14,\tfrac13\right), \ldots\right\}\] provides the
L\"uroth series \cite{lueroth1883:ueber_eine_eindeutige}.
\end{exa}

Under some mild restrictions one can replace the intervals by
arbitrary ones in order to get the generalized L\"uroth series
(\textit{cf.} Chapter 2.3 of
\cite{dajani_kraaikamp2002:ergodic_theory_numbers}).

Before pulling over the definition of normal and non-normal numbers,
we note that in contrast to the survey of Barat \textit{et
  al.}~\cite{barat_berthe_liardet+2006:dynamical_directions_in} we did
not use fibered systems for the definition of the dynamical
system. The reason lies in the concrete treatment of the border in the
case of Markov partitions. In partiuclar, when considering these
partitions it is clear that the sets $\mathcal{P}_i$ are all open
sets, whereas this is a priory not clear in the definition of fibered
systems by Schweiger \cite{schweiger1995:ergodic_theory_fibred}. This
plays a key role in the following analysis of the one-to-one
correspondence between the infinite word and the corresponding element
of $M$. By the definition of a Markov partition we have that every
$\omega\in X_{\mathcal{P},\phi}$ maps to a unique element $x\in
M$. However, the converse need not be true. Let us consider the
continued fractions expansion (Exemple
\ref{example:continued_fractions}). Then the rational $\frac14$ has
two expansions, namely $[4]$ and $[3,1]$. One observes that this
ambiguity originates from the intersections
$\overline{P_i}\cap\overline{P_j}$ for $i\neq j$ (which means from the
borders of $\overline{P_i}$). Thus we concentrate on the inner points,
which provide us with an infinite and unique expansion. Let
\[
U=\bigcup_{i=1}^\infty P_i,
\]
which is an open and dense ($\overline{U}=M$) set. Then for each
$n\geq1$ the set
\[
U_n=\bigcap_{k=0}^{n-1}\phi^{-k}(U),
\]
is open and dense in $M$. Thus by the Baire Category Theorem, the set
\begin{gather}\label{mani:U-infty}
U_\infty=\bigcap_{n=1}^\infty U_n
\end{gather}
is dense. Since $M\setminus U_\infty$ is the countable union of
nowhere dense sets it suffices to show that a set is residual in
$U_\infty$ in order to show that in fact it is residual in
$M$.

Since the definition of normal and thus non-normal numbers involves
the expansions of the elements in $M$ we need the map
$\pi_{\mathcal{P},\phi}:X_{\mathcal{P},\phi}\to M$ defined by
\[
\{\pi_{\mathcal{P},\phi}(\omega)\}=\bigcap_{n=1}^\infty\overline{D_n(\omega)}.
\]
Since $\pi_{\mathcal{P},\phi}$ is bijective on $U_\infty$ we may call
$\omega$ \textbf{the} symbolic expansion of $x$ if
$\pi_{\mathcal{P},\phi}(\omega)=x$. Thus in the following we will
silently suppose that $x\in U_\infty$.

After defining the environment we want to pull over the definitions of
normal and non-normal numbers to the symbolic dynamical system. To this end let
$\mathbf{b}\in\N^k$ be a block of letters of length $k$ and
$\omega=a_1a_2a_3\ldots\in X_{\mathcal{P},\phi}$ be the symbolic
representation of an element. Then we write
\[
\mathrm{P}(\omega,\mathbf{b},n)=\lvert\left\{0\leq i<n:a_{i+1}=b_1,\ldots,a_{i+k}=b_k\right\}\rvert
\]
for the frequency of the block $\mathbf{b}$ among the first $n$ letters of
$\omega$. In the same manner as above let 
\[
\mathrm{P}_k(\omega,n)=\left(\mathrm{P}(\omega,\mathbf{b}),n\right)_{\mathbf{b}\in\N^k}
\]
be the vector of all frequencies of blocks $\mathbf{b}$ of length $k$ among the
first $n$ letters of $\omega$.

Let $\mu$ be a given $\phi$-invariant probability measure on $X$ and
$\omega\in X$. Then we call the measure $\mu$ associated to $\omega$
if there exists a infinite sub-sequence $F$ of $\N$ such that for any
block $\mathbf{b}\in\Sigma^k$
$$\lim_{\substack{n\to\infty\\ n\in
    F}}\mathrm{P}(\omega,\mathbf{b},n)=\mu([\mathbf{b}]).$$ Furthermore, we
call $\omega$ a generic point for $\mu$ if we can take $F=\N$: then
$\mu$ is the only measure associated with $\omega$. If $\mu$ is the
maximal measure, then we call $\omega$ normal.

An application of Birkhoff's ergodic theorem yields for $\mu$ being
ergodic that almost all numbers $\omega\in X_{\mathcal{P},\phi}$ are
normal. In both Examples \ref{example:continued_fractions} and
\ref{example:lueroth} we have that the system is intrinsically
ergodic, which means that there exists a unique maximal ergodic
measure $\mu$ (\textit{cf.}  Chapter 3.1.2 of
\cite{dajani_kraaikamp2002:ergodic_theory_numbers}). 

Now we turn our attention to non-normal numbers. As in the paper of
Hyde \textit{et al.}  \cite{hyde10:_iterat_cesar_baire} we extend our
considerations to Ces\`aro averages of the frequencies. The idea
behind is that if the vector of frequencies has no accumulation point,
his Ces\`aro average might have one. To this end for a fixed block
$\mathbf{b}=b_1\ldots b_k\in\N^k$ let
\[\mathrm{P}^{(0)}(\omega,\mathbf{b},n)=\mathrm{P}(\omega,\mathbf{b},n).\]
For $r\geq1$ we recursively define
\[
\mathrm{P}^{(r)}(\omega,\mathbf{b},n)=\frac{\sum_{j=1}^n\mathrm{P}^{(r-1)}(\omega,\mathbf{b},j)}{n}
\]
to be the $r$th iterated Ces\`aro average of the frequency of the
block of digits $\mathbf{b}$ under the first $n$ digits. Furthermore
we define by
\[
\mathrm{P}_k^{(r)}(\omega,n):=\left(\mathrm{P}^{(r)}(\omega,\mathbf{b},n)\right)_{\mathbf{b}\in\N^k}
\]
the vector of $r$th iterated Ces\`aro averages. As above we are
interested in the accumulation points. Thus similar to above let
$\mathrm{A}^{(r)}_k(\omega)$ denote the set of accumulation points of the sequence
$(\mathrm{P}^{(r)}_k(\omega,n))_n$ with respect to
$\lVert\cdot\rVert_1$,\textit{i.e.}
\[
\mathrm{A}^{(r)}_k(\omega):=\left\{\mathbf{p}\in\Delta_k:\mathbf{p}\text{ is an
    accumulation point of }(\mathrm{P}^{(r)}_k(\omega,n))_n\right\}.
\]

We will denote the set of extremely non-k-normal numbers of $M$ by
$\mathbb{E}^{(0)}_k$. 
Similarly for $r\geq1$ and $k\geq1$ we denote by $\mathbb{E}^{(r)}_k$
the set of $r$th iterated Ces\`aro extremely non-k-normal numbers of
$M$. 
Furthermore for $r\geq1$ we denote by $\mathbb{E}^{(r)}$ the set of
$r$th iterated Ces\`aro extremely non-normal numbers and by
$\mathbb{E}$ the set of completely Ces\`aro extremely non-normal
numbers, \textit{i.e.}
\[
\mathbb{E}^{(r)}=\bigcap_{k}\mathbb{E}_k^{(r)}
\quad\text{and}\quad
\mathbb{E}=\bigcap_{r}\mathbb{E}^{(r)}=\bigcap_{r,k}\mathbb{E}_k^{(r)}.
\]

Then our result is the following
\begin{thm}\label{thm:extremely_non_normal}
  Let $k\geq1$ and $r\geq0$ be integers. Furthermore let
  $\mathcal{P}=\{P_1,P_2,\ldots\}$ be an infinite Markov partition for
  $(M,\phi)$. Suppose that the generated shift space
  $X_{\mathcal{P},\phi}$ is the one-sided full-shift. Then the set
  $\mathbb{E}^{(r)}_k$ is residual.
\end{thm}

\begin{rem}
  We note that the same holds true for $X_{\mathcal{P},\phi}$ being a
  one-sided shift of finite type. In fact the only change is a
  replacement of the definition of $Z_n$ and of Lemma
  \ref{lem:Znotempty} (\textit{cf.} Olsen
  \cite{olsen2003:multifractal_analysis_divergence} and Olsen and
  Winter
  \cite{olsen_winter2007:multifractal_analysis_divergence}).
\end{rem}



After considering extremely non-normal numbers we want to turn our
attention towards numbers with maximal oscillation frequency.
Similarly to above, for $r\geq0$, we denote by
$\mathrm{A}^{(r)}(\omega,\mathbf{b})$ the set of all accumulation
points of the sequence
$\left(\mathrm{P}^{(r)}(\omega,\mathbf{b},n)\right)_{n\in\N}$. Furthermore
we set
$$\mathrm{A}^{(r)}(\mathbf{b})=\bigcup_{\omega\in
  U_\infty}\mathrm{A}^{(r)}(\omega,\mathbf{b}).$$ Then the set of
numbers with $r$-th iterated Ces\`aro maximal frequency oscillation
$\mathbb{F}^{(r)}$ is defined by
$$\mathbb{F}^{(r)}=\left\{\omega\in
  U_\infty\colon\mathrm{A}^{(r)}(\mathbf{b})=\mathrm{A}^{(r)}(\omega,\mathbf{b})\text{
    for all }\mathbf{b}\in\N^*\right\}.$$

Our result is a generalization of Theorem 1 of Liao, Ma and Wang
\cite{liao_ma_wang2008:dimension_some_non}.
\begin{thm}\label{thm:maximal_frequency_oscillation}
Let $r\geq0$ be an integer and let $\mathcal{P}=\{P_1,P_2,\ldots\}$ be
an infinite Markov partition for $(M,\phi)$. Suppose that the
generated shift space $X_{\mathcal{P},\phi}$ is the one-sided
full-shift. Then $\mathbb{F}^{(r)}$ is residual.
\end{thm}

In the subsequent sections we will proof the two Theorems
\ref{thm:extremely_non_normal} and
\ref{thm:maximal_frequency_oscillation}. We will start with
a general section on properties of words which we need in the
proof of Theorem \ref{thm:extremely_non_normal} and which are
interesting on their own. Then we will show Theorem
\ref{thm:extremely_non_normal}. Finally in
Section~\ref{sec:proof-frequency-oscillation} we will prove Theorem
\ref{thm:maximal_frequency_oscillation} by showing that $\mathbb{E}^{(r)}\subset\mathbb{F}^{(r)}$.

\section{Preliminaries on words}
First of all we want to reduce the infinite problem to a finite
one. Thus instead of considering $\mathrm{S}_k$ as such we concentrate
on those probability vectors that only put weight on a finite set of
digits. Thus let
\begin{gather}\label{SkN}
\mathrm{S}_{k,N}=\left\{(p_{\mathbf{i}})_{\mathbf{i}\in\N^k}:
\begin{gathered}
p_{\mathbf{i}}\geq0,\sum_{\mathbf{i}\in\N^k}p_{\mathbf{i}}=1,
\sum_{i\in\N}p_{i\mathbf{i}}=\sum_{i\in\N}p_{\mathbf{i}i}\text{ for
  all }\mathbf{i}\in\N^{k-1}\\
p_{\mathbf{i}}=0\text{ for }\mathbf{i}\in\N^k\setminus\{1,\ldots,N\}^k
\end{gathered}\right\}
\end{gather}
be the set of shift invariant probability vectors, where only the
first $N$ digits are weighted. Furthermore let
\begin{gather}\label{Skstar}
\mathrm{S}_k^*=\bigcup_{N\geq1}\mathrm{S}_{k,N}
\end{gather}
be the union of all probability vectors over a finite alphabet.

Since $\mathrm{S}_k^*$ is a dense and separable subset of
$\mathrm{S}_k$, we may concentrate on a dense sequence
$(\mathbf{q}_{k,m})_m$ in $\mathrm{S}_k^*$. We fix
$\mathbf{q}=\mathbf{q}_{k,m}$ throughout the rest of this
section. Then $\mathbf{q}\in\mathrm{S}_{k,N}$ for some $N\geq1$, such
that $\mathbf{q}_{\mathbf{i}}=0$ for
$\mathbf{i}\in\N^k\setminus\{1,\ldots,N\}^k$. For $n\geq1$ we put
\[
Z_n=Z_n(\mathbf{q},N,k)=\left\{\omega\in\bigcup_{\ell\geq knN^k}\{1,\ldots,N\}^\ell\vert
\lVert\mathrm{P}_k(\omega)-\mathbf{q}\rVert_1\leq\frac1n\right\}.
\]
Since $\mathbf{q}$, $N$ and $k$ will be fixed we may omit them throughout the
rest of this section.

The main idea consists now in the construction of a word having the
desired frequencies. In particular, for a given word $\omega$ we want
to show that we can add sufficiently many copies of any word from
$Z_n$ to get a word with the desired properties. To this end we first
need, that there is at least one word in $Z_n$, \textit{i.e.} that
$Z_n$ is not empty.

\begin{lem}[{\cite[Lemma 2.4]{olsen2003:extremely_non_normal}}]\label{lem:Znotempty}
  For all $n\geq1$, $\mathbf{q}\in\mathrm{S}_k^*$, $N\in\N$ and
  $k\in\N$ we have $Z_n(\mathbf{q},N,k)\not =\emptyset$.
\end{lem}

Now we may construct our word by adding arbitrary many copies of an
element of $Z_n$.

\begin{lem}\label{lem:frequency:word}
  Let $N,n,t$ be positive integers and
  $\mathbf{q}\in\mathrm{S}_{k,N}$. Furthermore let
  $\omega=\omega_1\ldots\omega_t\in\N^t$ be a word of length $t$ and
  let $M=\max_{1\leq i\leq t}\omega_i$ be the maximal ``digit'' in
  $\omega$. Then, for any $\gamma\in Z_n(\mathbf{q},N,k)$ and any
\begin{gather}\label{m:bound}
\ell\geq L:=t+\lvert\gamma\rvert\max\left(n,\frac tk\max\left(1,\frac{M^k}{N^k}\right)\right)
\end{gather}
we get that
\[
\lVert\mathrm{P}_k(\omega\gamma^*,\ell)-\mathbf{q}\rVert_1\leq\frac6n.
\]
\end{lem}

\begin{proof}
We set $s:=\lvert\gamma\rvert$ and 
\[
\sigma=\omega\gamma^*\vert\ell.
\]
Furthermore we set $q$ and $0\leq r<s$ such that $m=t+q s+r$. Since an
occurrence can happen in $\omega$, in $\gamma$, somewhere in between or at the
end, for every $\mathbf{i}\in\N^k$ we clearly have that
\[
\frac{qs}{\ell}\mathrm{P}(\gamma,\mathbf{i})\leq
\mathrm{P}(\sigma,\mathbf{i})\leq\frac{qs\mathrm{P}(\gamma,\mathbf{i})}{\ell}+\frac{t+q(k-1)+r}{\ell}.
\]

Now we concentrate on the occurrences in multiples of $\gamma$ and show that we may
neglect those outside of $\gamma$, \textit{i.e.},
\[
\lVert\mathrm{P}_k(\sigma)-\mathbf{q}\rVert_1
\leq\lVert\mathrm{P}_k(\sigma)-\frac{qs}\ell \mathrm{P}_k(\gamma)\rVert_1+
\lVert\frac{qs}\ell \mathrm{P}_k(\gamma)-\mathbf{q}\rVert_1.
\]
We will estimate both parts separately. For the first one we get that
\begin{align*}
\lVert\mathrm{P}_k(\sigma)-\frac{qs}\ell
\mathrm{P}_k(\gamma)\rVert_1
&=\sum_{\mathbf{i}\in\{1,\ldots,N\}^k}\lvert\mathrm{P}(\sigma,\mathbf{i})-\frac{qs}\ell
\mathrm{P}(\gamma,\mathbf{i})\rvert+\sum_{\mathbf{i}\in\N^k\setminus\{1,\ldots,N\}^k}\lvert\mathrm{P}(\sigma,\mathbf{i})-\frac{qs}\ell
\mathrm{P}(\gamma,\mathbf{i})\rvert\\
&\leq\sum_{\mathbf{i}\in\{1,\ldots,N\}^k}\frac{t+qk+s}{\ell}+
\sum_{\substack{\mathbf{i}\in\N^k\setminus\{1,\ldots,N\}^k\\\mathrm{P}(\omega,\mathbf{i})\neq0}}\frac{t}{\ell}\\
&\leq N^k\frac{t+qk}{qnkN^k}+\frac1q+M^k\frac{t}{qnkN^k}\\
&=\frac1n+(c+1)\frac1q,
\end{align*}
where we have used that $\ell\geq qs\geq qnkN^k$ and written
\[
c=\frac{t}{nk}\left(1+\frac{M^k}{N^k}\right).
\]

For the second part we get that
\begin{align*}
\lVert\frac{qs}\ell \mathrm{P}_k(\gamma)-\mathbf{q}\rVert_1
&\leq\lVert\frac{qs}\ell\mathrm{P}_k(\gamma)-\mathrm{P}_k(\gamma)\rVert_1
+\lVert\mathrm{P}_k(\gamma)-\mathbf{q}\rVert_1\\
&\leq qs\lvert\frac1\ell-\frac1{qs}\rvert+\frac1n\\
&\leq \frac t\ell+\frac1n\leq\frac t{qnkN^k}+\frac1n.
\end{align*}

Putting these together yields
\[
\lVert\mathrm{P}_k(\sigma)-\mathbf{q}\rVert_1\leq
\frac1n+(c+1)\frac1q+\frac t{qnkN^k}+\frac1n.
\]

By our assumptions on the size of $\ell$ in \eqref{m:bound} this proves the
lemma.
\end{proof}


\section{Proof of Theorem \ref{thm:extremely_non_normal}}\label{sec:proof-extremely-non-normal}
The standard method of proof is to construct a subset $E$ of
$\mathbb{E}_k^{(r)}$ which is easier to handle and already
residual. In our construction of the set $E$ we mainly follow the
ideas of Hyde \textit{et al.} \cite{hyde10:_iterat_cesar_baire}. We
start by recursively defining the functions $\varphi_m$ for $m\ge1$ by
$\varphi_1(x)=2^x$ and $\varphi_m(x)=\varphi_1(\varphi_{m-1}(x))$ for
$m\geq2$. Furthermore we set
$\mathbb{D}=(\mathbb{Q}^{\left(\N^k\right)}\cap\mathrm{S}^*_k)$. Since
$\mathbb{D}$ is countable and dense in $\mathrm{S}^*_k$ and therefore
dense in $\mathrm{S}_k$ we may concentrate on the probability vectors
$\mathbf{q}\in\mathbb{D}$.

Now we say that a sequence $(\mathbf{x}_n)_n$ in $\R^{\left(\N^k\right)}$ has property $P$ if
for all $\mathbf{q}\in\mathbb{D}$, $m\in\N$, $i\in\N$, and
$\varepsilon>0$, there exists a $j\in\N$ satisfying:
\begin{enumerate}
\item $j\geq i$,
\item $j/2^j<\varepsilon$,
\item if $j<n<\varphi_m(j)$ then
  $\lVert\mathbf{x}_n-\mathbf{q}\rVert_1<\varepsilon$.
\end{enumerate}

Then we define our set $E$ to consist of all frequency vectors having property
$P$, \textit{i.e.} 
\[
E=\{x\in U_\infty:\text{ $(\mathrm{P}^{(0)}_k(x;n))_{n=1}^\infty$ has
  property $P$}\}.
\]

We will proceed in three steps showing that
\begin{enumerate}
\item $E$ is residual, 
\item if $(\mathrm{P}^{(r)}(x;n))_{n=1}^\infty$ has property $P$, then
  also $(\mathrm{P}^{(r+1)}(x;n))_{n=1}^\infty$ has property $P$, and
\item $E\subseteq\mathbb{E}^{(r)}_k$.
\end{enumerate}

\begin{lem}\label{lem:Eresidual}
The set $E$ is residual.
\end{lem}

\begin{proof}
  For fixed $h,m,i\in\N$ and $\mathbf{q}\in\mathbb{D}$, we say that a
  sequence $(\mathbf{x}_n)_n$ in $\R^{\N^k}$ has property
  $P_{h,m,\mathbf{q},i}$ if for every $\varepsilon>1/h$, there exists
  $j\in\N$ satisfying:
\begin{enumerate}
\item $j\geq i$,
\item $j/2^j<\varepsilon$,
\item if $j<n<\varphi_m(2^j)$, then $\lVert\mathbf{x}_n-\mathbf{q}\rVert_1<\varepsilon$.
\end{enumerate}
Now let $E_{h,m,\mathbf{q},i}$ be the set of all points whose
frequency vector satisfies property $P_{h,m,\mathbf{q},i}$,
\textit{i.e.}
\[
E_{h,m,\mathbf{q},i}:=\left\{x\in U_{\infty}:\left(\mathrm{P}^{(0)}_k(x;n)\right)_{n=1}^\infty\text{ has
  property }P_{h,m,\mathbf{q},i}\right\}.
\]
Obviously we have that
\[
E=\bigcap_{h\in\N}\bigcap_{m\in\N}\bigcap_{\mathbf{q}\in\mathbb{D}}\bigcap_{i\in\N}E_{h,m,\mathbf{q},i}.
\]

Thus it remains to show, that $E_{h,m,\mathbf{q},i}$ is open and dense.
\begin{enumerate}
\item\textbf{$E_{h,m,\mathbf{q},i}$ is open.} Let $x\in
  E_{h,m,\mathbf{q},i}$, then there exists a $j\in\N$ such that $j\geq
  i$, $j/2^j<1/h$, and if $j<n<\varphi_m(2^j)$, then
  \[
  \lVert\mathrm{P}^{(1)}_k(x;n)-\mathbf{q}\rVert_1<1/h.
  \]

  Let $\omega\in X$ be such that $x=\pi(\omega)$ and set
  $t:=\varphi_m(2^j)$. Since $D_{t}(\omega)$ is open, there exists a
  $\delta>0$ such that the ball $B(x,\delta)\subseteq
  D_{t}(\omega)$. Furthermore, since by definition all $y\in
  D_{t}(\omega)$ have their first $t$ digits the same as $x$, we get
  that
  \[
  B(x,\delta)\subseteq D_{t}(\omega)\subseteq E_{h,m,\mathbf{q},i}.
  \]

\item\textbf{$E_{h,m,\mathbf{q},i}$ is dense.} Let $x\in U_\infty$ and
  $\delta>0$. We must find $y\in B(x,\delta)\cap E_{h,m,\mathbf{q},i}$.
  
  Let $\omega\in X$ be such that $x=\pi(\omega)$. Since
  $\diam\overline{D_t(\omega)}\to0$ for $t\to\infty$ and $x\in
  D_t(\omega)$ for $t\geq1$ there exists a $t$ such that $D_t(\omega)\subset
  B(x,\delta)$. Let $\sigma=\omega\vert t$ be the first $t$ digits of
  $x$.

  Now, an application of Lemma \ref{lem:Znotempty} with $n=6h$ yields that there exists a finite
  word $\gamma$ such that
  \[
    \lVert\mathrm{P}_k(\gamma)-\mathbf{q}\rVert_1\leq\frac1{6h}.
  \]
  
  Let $\varepsilon\geq\frac1h$ and $L$ be as in the statement of Lemma
  \ref{lem:frequency:word}. Then we choose $j$ such that
  \[
  \frac{j}{2^j}<\varepsilon\quad\text{and}\quad j\geq \max\left(L,i\right).
  \]
  An application of Lemma \ref{lem:frequency:word} with $n=6h$ then gives us that
  \[
    \lVert\mathrm{P}_k(\sigma\gamma^*\vert j)-\mathbf{q}\rVert_1\leq\frac6n=\frac1h\leq\varepsilon.
  \]

  Thus we choose $y\in D_{j}(\sigma\gamma^*)$. Then on the one hand
  $y\in D_{j}(\sigma\gamma^*)\subset D_{t}(\omega)\subset B(x,\delta)$ and on the other
  hand $y\in D_{j}(\sigma\gamma^*)\subset E_{h,m,\mathbf{q},i}$

\end{enumerate}
It follows that $E$ is the countable intersection of open and dense
sets and therefore $E$ is residual in $U_\infty$.
\end{proof}

\begin{lem}\label{lem:propertyP}
Let $\omega\in X_{\mathcal{P},\phi}$. If
$(\mathrm{P}^{(r)}(\omega,n))_{n=1}^\infty$ has property $P$, then
also $(\mathrm{P}^{(r+1)}(\omega,n))_{n=1}^\infty$ has property $P$.
\end{lem}

This is Lemma 2.2 of \cite{hyde10:_iterat_cesar_baire}. However, the
proof is short so we present it here for completeness.

\begin{proof}
Let $\omega\in X_{\mathcal{P},\phi}$ be such that
$(\mathrm{P}_k^{(r)}(\omega;n))_{n=1}^\infty$ has property 
$P$, and fix $\varepsilon>0,\mathbf{q}\in\mathbb{D}$, $i\in\N$ and
$m\in\N$. Since $(\mathrm{P}_k^{(r)}(\omega,n))_{n=1}^\infty$ has property
$P$, there exists $j'\in\N$ with $j'\geq i$,
$j'/2^{j'}<\varepsilon/3$, and such that for
$j'<n<\varphi_{m+1}(2^{j'})$ we have that
$\lVert\mathrm{P}_k^{(r)}(\omega,n)-\mathbf{q}\rVert_1<\varepsilon/3$.

We set $j=2^{j'}$ and show that
$(\mathrm{P}_k^{(r+1)}(\omega,n))_{n=1}^\infty$ has property $P$ with this
$j$. For all $j<n<\varphi_m(2^j)$ (\textit{i.e.}
$2^{j'}<n<\varphi_{m+1}(2^{j'})$), we have
\begin{align*}
\lVert\mathrm{P}_k^{(r+1)}(\omega,n)-\mathbf{q}\rVert_1
&=\lVert\frac{\mathrm{P}_k^{(r)}(\omega,1)+\mathrm{P}_k^{(r)}(\omega,2)+\cdots+\mathrm{P}_k^{(r)}(\omega,n)}{n}-\mathbf{q}\rVert_1\\
&=\lVert\frac{\mathrm{P}_k^{(r)}(\omega,1)+\mathrm{P}_k^{(r)}(\omega,2)+\cdots+\mathrm{P}_k^{(r)}(\omega,j')}{n}\right.\\
&\quad+\left.\frac{\mathrm{P}_k^{(r)}(\omega,j'+1)+\mathrm{P}_k^{(r)}(\omega,j'+2)+\cdots+\mathrm{P}_k^{(r)}(\omega,n)-(n-j')\mathbf{q}}{n}-\frac{j'\mathbf{q}}{n}\rVert_1\\
&\leq\frac{\lVert\mathrm{P}_k^{(r)}(\omega,1)+\mathrm{P}_k^{(r)}(\omega,2)+\cdots+\mathrm{P}_k^{(r)}(\omega,j')\rVert_1}{n}\\
&\quad+\frac{\lVert\mathrm{P}_k^{(r)}(\omega,j'+1)-\mathbf{q}\rVert_1+\cdots+\lVert\mathrm{P}_k^{(r)}(\omega,n)-\mathbf{q}\rVert_1}{n}
  -\frac{\lVert j'\mathbf{q}\rVert_1}{n}\\
&\leq\frac{j'}{n}+\frac{\varepsilon}3\frac{n-j'}{n}+\frac{j'}{n}
\leq\frac{j'}{2^{j'}}+\frac{\varepsilon}3+\frac{j'}{2^{j'}}
\leq\frac{\varepsilon}3+\frac{\varepsilon}3+\frac{\varepsilon}3
=\varepsilon.
\end{align*}
\end{proof}

\begin{lem}\label{lem:Esubset}
The set $E$ is a subset of $\mathbb{E}^{(r)}_k$.
\end{lem}

\begin{proof}
We will show, that for any $x\in E$ we also have
$x\in\mathbb{E}^{(r)}_k$. To this end, let $x\in E$ and $\omega\in
X_{\mathcal{P},\phi}$ be the symbolic expansion of $x$, \textit{i.e.} $x=\pi(\omega)$. Since
$(\mathrm{P}_k^{(0)}(\omega,n))_n$ has property $P$, by iterating Lemma \ref{lem:propertyP}
we get that $(\mathrm{P}_k^{(r)}(\omega,n)$ has property $P$.

Thus it suffices to show that $\mathbf{p}$ is an accumulation point of
$(\mathrm{P}_k^{(r)}(\omega,n))_n$ for any $\mathbf{p}\in\mathrm{S}_k$. Therefore
we fix $h\in\N$ and, since $\mathbb{D}$ is dense in $\mathrm{S}_k$, we find a $\mathbf{q}\in\mathbb{D}$ such that
\[
\lVert\mathbf{p}-\mathbf{q}\rVert_1<\frac1h.
\]
Since $(\mathrm{P}_k^{(r)}(\omega,n))_n$ has property $P$ for any $m\in\N$ we
find $j\in\N$ with $j\geq h$ and such that if $j<n<\varphi_m(2^j)$ then
$\lVert\mathrm{P}_k^{(r)}(\omega,n)-\mathbf{q}\rVert_1<\frac1h$. Hence let $n_h$
be any integer with $j<n_h<\varphi_m(2^j)$, then
\[
\lVert\mathrm{P}_k^{(r)}(\omega,n_h)-\mathbf{q}\rVert_1<\frac1h.
\]

Thus each $n_h$ in the sequence $(n_h)_h$ satisfies
\[
\lVert\mathbf{p}-\mathrm{P}_k^{(r)}(\omega,n_h)\rVert_1
\leq\lVert\mathbf{p}-\mathbf{q}\rVert_1+\lVert\mathrm{P}_k^{(r)}(\omega,n_h)-\mathbf{q}\rVert_1<\frac2h.
\]
Since $n_h>h$ we may extract an increasing sub-sequence $(n_{h_u})_u$ such that
$\mathrm{P}_k^{(r)}(\omega,n_{h_u})\rightarrow\mathbf{p}$ for
$u\rightarrow\infty$. Thus $\mathbf{p}$ is an accumulation point of
$\mathrm{P}_k^{(r)}(\omega,n)$, which proves the lemma.
\end{proof}

\begin{proof}[Proof of Theorem  \ref{thm:extremely_non_normal}]
Since by Lemma \ref{lem:Eresidual} $E$ is residual in $U_\infty$ and
by Lemma \ref{lem:Esubset} $E$ is a subset of $\mathbb{E}^{(r)}_k$ we get that
$\mathbb{E}^{(r)}_k$ is residual in $U_\infty$. Again we note that
$M\setminus U_\infty$ is the countable union of nowhere dense sets and
therefore $\mathbb{E}^{(r)}_k$ is also residual in $M$. 
\end{proof}
 
\section{Proof of Theorem \ref{thm:maximal_frequency_oscillation}}\label{sec:proof-frequency-oscillation}
Following the proof of Liao, Ma and Wang
\cite{liao_ma_wang2008:dimension_some_non} it suffices to show that
$$\mathbb{E}^{(r)}\subset\mathbb{F}^{(r)}.$$

First the following lemma provides us with a suitable definition of
$\mathrm{A}^{(r)}(\omega,\mathbf{b})$.

\begin{lem}\label{lem:interval}
Let $r\geq0$ be an integer, $\omega\in X_{\mathcal{P},\phi}$ and
$\mathbf{b}\in\N^*$. Then
$$\mathrm{A}^{(r)}(\omega,\mathbf{b})=\left[\liminf_{n\to\infty}
  P^{(r)}(\omega,\mathbf{b},n),
\limsup_{n\to\infty} P^{(r)}(\omega,\mathbf{b},n)\right].$$
\end{lem}

\begin{proof}
It suffices to show that the gaps between two consecutive frequencies
tend to zero, \textit{i.e.}
$$\lim_{n\to\infty}\left(\mathrm{P}^{(r)}(\omega,\mathbf{b},n+1)-\mathrm{P}^{(r)}(\omega,\mathbf{b},n)\right)=0.$$

For $r=0$ a direct upper and lower estimate for the number of
occurrences yields
$$\lvert\mathrm{P}^{(0)}(\omega,\mathbf{b},n+1)-\mathrm{P}^{(0)}(\omega,\mathbf{b},n)\rvert\leq\frac1{n+1}.$$

Since, for $i,j\geq1$ and $r\geq0$, we have that $$\lvert
\mathrm{P}^{(r)}(\omega,\mathbf{b},i)-\mathrm{P}^{(r)}(\omega,\mathbf{b},j)\rvert\leq1,$$
we get by the definition of $\mathrm{P}^{(r+1)}(\omega,\mathbf{b},n)$ that
\begin{align*}
\lvert
\mathrm{P}^{(r+1)}(\omega,\mathbf{b},n+1)-\mathrm{P}^{(r+1)}(\omega,\mathbf{b},n)\rvert
\leq\frac{\sum_{j=1}^n\lvert
  \mathrm{P}^{(r)}(\omega,\mathbf{b},n+1)-\mathrm{P}^{(r)}(\omega,\mathbf{b},j)\rvert}{(n+1)n}\leq\frac1{n+1}.
\end{align*}
\end{proof}

Let $\mathbf{b}=b_1b_2\ldots b_k\in\N^*$ be a word of length
$k$. Then we denote by $\mathrm{per}(\mathbf{b})$ the basic period of
$\mathbf{b}$, \textit{i.e.}
$$\mathrm{per}(\mathbf{b}):=\min\{p\leq k\colon b_{p+j}=b_j\text{ for
}1\leq j\leq k-p\}.$$ Furthermore we call the factor
$\widetilde{\mathbf{b}}:=b_1\ldots b_{\mathrm{per}(\mathbf{b})}$ the basic
factor. Then we have the following
\begin{lem}\label{lem:basic_factor}
  Let $r\geq0$ be an integer and $\mathbf{b}\in\N^*$ be a finite word
  with basic period $p$ and basic factor
  $\widetilde{\mathbf{b}}=b_1\ldots b_p$. Then, for each
  $n\geq2$, $$\lim_{n\to\infty}P^{(r)}(\widetilde{\mathbf{b}}^*,\mathbf{b},n)=\frac1p.$$
\end{lem}

\begin{proof}
  For $r=0$ this is Lemma 2.2 of
  \cite{liao_ma_wang2008:dimension_some_non}. The case $r\geq1$
  follows,
  since $$\lim_{n\to\infty}P^{(r)}(\widetilde{\mathbf{b}}^*,\mathbf{b},n)
  =\lim_{n\to\infty}P^{(r-1)}(\widetilde{\mathbf{b}}^*,\mathbf{b},n)=\cdots=
  \lim_{n\to\infty}P^{(0)}(\widetilde{\mathbf{b}}^*,\mathbf{b},n)=\frac1p.$$
\end{proof}

Now we have enough tools to state the proof of Theorem
\ref{thm:maximal_frequency_oscillation}.

\begin{proof}[Proof of Theorem \ref{thm:maximal_frequency_oscillation}]
  Let $\omega\in\mathbb{E}^{(r)}$ and $\mathbf{b}=b_1\ldots
  b_k\in\N^*$ be a finite word with basic period $p$. Then
  Lemma~\ref{lem:interval} and Lemma~\ref{lem:basic_factor} imply
  that $$\mathrm{A}^{(r)}(\mathbf{b})=\left[0,\tfrac1p\right].$$
  Therefore in order to prove that $\omega\in\mathbb{F}^{(r)}$ it
  suffices to show that $0$ and $\tfrac1p$ are limit points of
  $(\mathrm{P}^{(r)}(\omega,\mathbf{b},n))_{n\in\N}$. Furthermore,
  since $\omega\in\mathbb{E}$, for any $\varepsilon>0$ and
  $\mathbf{q}\in\mathrm{S}_k$ we have
  $\lVert\mathrm{P}_k^{(r)}(\omega,n))-\mathbf{q}\rVert_1$ for
  infinitely many $n$. Thus it suffices to find two suitable
  probability vectors $\mathbf{q}$ providing the limit points $0$ and
  $\frac1p$ for
  $\left(\mathrm{P}^{(r)}(\omega,\mathbf{b},n)\right)_{n\in\N}$.

\begin{itemize}
\item \textbf{$0$ is limit point.} We chose a digit $d$ which is
bigger than any digit in $\mathbf{b}$, \textit{i.e.}
$d>\max\left\{w_i\colon1\leq i\leq k\right\}$. Then we define the probability vector
$\mathbf{q}=(q_\mathbf{i})_{\mathbf{i}\in\N^k}$ by
\begin{gather*}
q_\mathbf{i}=\begin{cases}
1&\text{if }\mathbf{i}=\underbrace{d\ldots d}_{k\text{ times}},\\
0&\text{else}.
\end{cases}
\end{gather*}
We clearly have that $\mathbf{q}\in\mathrm{S}_k$. Since
$\mathrm{P}^{(r)}(\omega,\mathbf{b},n))<\varepsilon$ infinitely often,
we have that $0$ is a limit point.

\item \textbf{$\tfrac1p$ is limit point.} We note that
$\gamma=\mathbf{b}^\infty$ is a periodic point with minimal period $p$
under the map $\phi$. Let $\mu$ be the periodic orbit measure, which
has mass $\frac1p$ at each of the points
$\{\gamma,\phi\gamma,\ldots,\phi^{p-1}\gamma\}$. Then $\mu$ is
shift-invariant and induces a shift-invariant probability vector
$$\mathbf{q}=(q_\mathbf{b})_{\mathbf{b}\in\N^k}=(\mu(\mathbf{b}) )_{\mathbf{b}\in\N^k}.$$
Since $\mathbf{q}\in\mathrm{S}_k$ and
$q_\mathbf{b}=\mu(\mathbf{b})=\frac1p$, we have
$\lvert\mathrm{P}^{(r)}(\omega,\mathbf{b},n)-\frac1p\rvert<\varepsilon$
infinitely often. Therefore $\frac1p$ is also a limit point.
\end{itemize}

\end{proof}

\providecommand{\bysame}{\leavevmode\hbox to3em{\hrulefill}\thinspace}
\providecommand{\MR}{\relax\ifhmode\unskip\space\fi MR }
\providecommand{\MRhref}[2]{%
  \href{http://www.ams.org/mathscinet-getitem?mr=#1}{#2}
}
\providecommand{\href}[2]{#2}

\end{document}